\documentclass[12pt,reqno]{amsart}
\usepackage{microtype}
\usepackage{tikz}
\usepackage{xcolor}
\usepackage{latexsym,amsmath,amssymb,epic}
\usepackage{latexsym,amssymb,epic}
\usepackage{amsmath,amsthm}
\usepackage{color}
\usepackage[bottom=3.5cm,left=3.5cm,right=3.5cm]{geometry}
\usepackage{etoolbox}
\usepackage{mathrsfs}
\usepackage[shortlabels]{enumitem}
\usepackage[noadjust]{cite}
\usepackage{hyperref}
\hypersetup{colorlinks=true, 
	linktoc=all, 
	linkcolor=blue} 

\allowdisplaybreaks[4]

\numberwithin{equation}{section}

\theoremstyle{plain}
\newtheorem{theorem}{Theorem}[section]
\newtheorem*{theorem*}{Theorem}

\newtheorem{lemma}[theorem]{Lemma}

\theoremstyle{definition}

\newtheorem*{example*}{Example}

\newtheorem*{conjecture*}{Conjecture}

\newtheorem*{remark*}{Remark}
\newtheorem*{remarks*}{Remarks}

\makeatletter
\patchcmd{\@settitle}{\uppercasenonmath}{\boldmath\uppercasenonmath}{}{}
\patchcmd{\section}{\scshape}{\bfseries\boldmath}{}{}
\patchcmd{\subsection}{\bfseries}{\bfseries\boldmath}{}{}
\renewcommand{\@secnumfont}{\bfseries}
\makeatother

\makeatletter
\renewcommand{\maketag@@@}[1]{\hbox{\normalsize\normalfont#1}}%
\@namedef{subjclassname@2020}{\textup{2020} Mathematics Subject Classification}

\makeatother

\begin{document}

\title[An elementary approach to general theta function identities]
{An elementary approach to Sun Kim's general theta function identities}

\author[D. Tang]{Dazhao Tang}

\address[Dazhao Tang]{School of Mathematical Sciences, Chongqing Normal University,
Chongqing 401331, P.R. China}
\email{dazhaotang@sina.com}

\date{\today}

\begin{abstract}
Ramanujan's modular equations of degrees $3$, $5$, $7$, $11$ and $23$ are closely
related to certain theta function identities. Warnaar generalized the identities
arising from the modular equations of degrees $3$ and $7$ to a general theta function
identity. Kim subsequently obtained a further general theta function identity
associated with the modular equations of degrees $5$, $11$ and $23$, and established
several general theta function identities containing known partition theorems as
special cases. In this paper, we present an elementary approach to Kim's general
theta function identities, yielding  $q$-series proofs of these identities within a
common framework.
\end{abstract}



\thanks{}

\maketitle

\section{Introduction}

Throughout this paper, we assume that $q$ is a complex number satisfying $|q|<1$ and
use the standard $q$-Pochhammer symbol
\begin{align*}
(a;q)_\infty &:=\prod_{j=0}^\infty(1-aq^j).
\end{align*}
For notational convenience, we write
\begin{align*}
\theta(a_1,a_2,\ldots,a_m;q)=\prod_{i=1}^m(a_i;q)_\infty(q/a_i;q)_\infty.
\end{align*}

Five modular equations of degrees $3$, $5$, $7$, $11$ and $23$ appear in Ramanujan's
notebooks. These remarkable identities were also discovered independently by
Schr\"{o}ter \cite{Rus1887} and R. Russell \cite{Rus1890, Sch1890}. In their equivalent
$q$-series forms, these modular equations can be written as
\begin{align}
(-q;q^2)_\infty^2(-q^3;q^6)_\infty^2 &-(q;q^2)_\infty^2(q^3;q^6)_\infty^2\nonumber\\
 &\qquad=4q(-q^2;q^2)_\infty^2(-q^6;q^6)_\infty^2,\label{Mod-eq-3}\\
(-q;q^2)_\infty^4(-q^5;q^{10})_\infty^4 &-(q;q^2)_\infty^4(q^5;q^{10})_\infty^4
\nonumber\\
 &\qquad=8q+16q^3(-q^2;q^2)_\infty^4(-q^{10};q^{10})_\infty^4,\label{Mod-eq-5}\\
(-q;q^2)_\infty(-q^7;q^{14})_\infty &-(q;q^2)_\infty(q^7;q^{14})_\infty\nonumber\\
 &\qquad=2q(-q^2;q^2)_\infty(-q^{14};q^{14})_\infty,\label{Mod-eq-7}\\
(-q;q^2)_\infty^2(-q^{11};q^{22})_\infty^2 &-(q;q^2)_\infty^2(q^{11};q^{22})_\infty^2
\nonumber\\
 &\qquad=4q+4q^3(-q^2;q^2)_\infty^2(-q^{22};q^{22})_\infty^2,\label{Mod-eq-11}\\
(-q;q^2)_\infty(-q^{23};q^{46})_\infty &-(q;q^2)_\infty(q^{23};q^{46})_\infty\nonumber\\
 &\qquad=2q+2q^3(-q^2;q^2)_\infty(-q^{46};q^{46})_\infty.\label{Mod-eq-23}
\end{align}
The identities \eqref{Mod-eq-3} and \eqref{Mod-eq-7} were established by Farkas and Kra
\cite{FK2000, FK2001} using theta function theory. Hirschhorn \cite{Hir2006}
subsequently gave an elementary $q$-series proof of \eqref{Mod-eq-3}.

Warnaar \cite{War2005} obtained a substantial generalization of \eqref{Mod-eq-3} and
\eqref{Mod-eq-7}. More precisely, he established the following general theta function
identity:
\begin{align}
\theta(-c,-ac,-bc,-abc;q) &-\theta(c,ac,bc,abc;q)\nonumber\\
 &\qquad\qquad=c\!\;\theta\big({-}1,-a,-b,-abc^2;q\big),\label{War-gene}
\end{align}
and provided three different proofs of \eqref{War-gene}, including a combinatorial
proof. Subsequently, Baruah and Berndt \cite{BB2007} observed that an equivalent form
of \eqref{War-gene} had already appeared in Ramanujan's Notebook; see
\cite[p.~47, Corollary]{Ber1991}. Kim \cite{Kim2021} subsequently obtained a further
generalization in the form of a theta function identity. In particular,
\eqref{Mod-eq-5}, \eqref{Mod-eq-11} and \eqref{Mod-eq-23} arise as special cases of her
identity, which we state below.

To streamline the statements of this and related theta function identities, we
introduce the following notation. Throughout this paper, set
\begin{align*}
x_1=c,\qquad x_2=d,\qquad x_3=\sqrt{cd},\qquad S(x)=\frac{abx^2}{q}.
\end{align*}
Define
\begin{align*}
A(x)=\theta\big({-}x,-ax,-bx,-abx;q^2\big).
\end{align*}
For a finite tuple $(u_1,u_2,\ldots,u_k)$, we further define
\begin{align*}
\mathscr{D}(u_1,u_2,\ldots,u_k) &=\prod_{j=1}^k A(u_j)-\prod_{j=1}^k A(-u_j),\\
\mathscr{S}(u_1,u_2,\ldots,u_k) &=\prod_{j=1}^k S(u_j).
\end{align*}

\begin{theorem}{\textup{\cite[Theorem~1.3]{Kim2021}}}\label{Main-THM-1}
Let $a,\!\:b,\!\:c,\!\:d,\!\:q\in\mathbb{C}\setminus\{0\}$ with $|q|<1$. Then
\begin{align}
\mathscr{D}\big(x_1^2,x_2^2,x_3^2\big) &-\mathscr{S}\big(x_1^2,x_2^2,x_3^2\big)
\mathscr{D}\big(qx_1^2,qx_2^2,qx_3^2\big)\nonumber\\
 &\qquad\qquad=2c^2\dfrac{\theta\big(a^2,b^2,d^3/c^3,a^2b^2c^4d^4/q^2;q^2\big)}
{\theta\big(a,b,d/c,abc^2d^2;q^2\big)}.\label{Kim-iden-1}
\end{align}
\end{theorem}

Kim's proof of \eqref{Kim-iden-1} relies on the theory of multiplicative elliptic
functions, together with a pole-cancellation argument and Liouville's theorem. At the
end of her paper, Kim suggested several directions for further investigation. In
particular, she \cite[p.~13]{Kim2021} remarked that it would be interesting to find
an alternative proof of \eqref{Kim-iden-1} using elementary $q$-series techniques or
combinatorial methods. We first address this question by providing an elementary
$q$-series proof of \eqref{Kim-iden-1}. Our proof combines Warnaar's identity
\eqref{War-gene} with Jacobi's triple product identity and the fundamental theta
identity of Weierstrass.

Very recently, Kim \cite{Kim2026} established further general theta function
identities using the theory of multiplicative elliptic functions, together with a
pole-cancellation argument and Liouville's theorem. As immediate applications, she
obtained several theta function identities previously derived or conjectured by Berndt
and R. Zhou \cite{BZ2015} and by Sandon and Zanello \cite{SZ2014} as special cases.
We next turn to several general theta function identities established by Kim \cite{Kim2026}. In contrast to the approach based on multiplicative elliptic functions
used in \cite{Kim2026}, we shall derive these identities by elementary $q$-series
methods. Together with \eqref{Kim-iden-1}, the proofs presented below provide a common
$q$-series approach to Kim's general theta function identities. We begin with the following two identities.

\begin{theorem}{\textup{\cite[Theorem~5.1]{Kim2026}}}\label{Main-THM-2}
Let $a,\!\:b,\!\:c,\!\:d,\!\:q\in\mathbb{C}\setminus\{0\}$ with $|q|<1$. Then
\begin{align}
\mathscr{D}(x_1,x_2) &+\mathscr{S}(x_1,x_2)\mathscr{D}(qx_1,qx_2)\nonumber\\
 &\quad=\dfrac{2abc^2d}{q}\theta(-a,-b,-d/c,-abcd;q).\label{Kim-iden-2}
\end{align}
\end{theorem}

\begin{theorem}{\textup{\cite[Theorem~2.1]{Kim2026}}}\label{Main-THM-3}
Let $a,\!\:b,\!\:c,\!\:d,\!\:q\in\mathbb{C}\setminus\{0\}$ with $|q|<1$. Then
\begin{align}
\mathscr{D}(x_1,qx_1,x_2) &-\mathscr{S}(x_2)\mathscr{D}(x_1,qx_1,qx_2)\nonumber\\
 &\quad=-\dfrac{2abcd^2}{q}\dfrac{\theta\big(a^2,b^2,a^2b^2c^4,abd^2q;q^2\big)}
{\theta\big(a,b,abc^2,abc^2q;q^2\big)}.\label{Kim-iden-3}
\end{align}
\end{theorem}

The next general theta function identity is of a form different from those of
\eqref{Kim-iden-1}--\eqref{Kim-iden-3}. We therefore introduce the following two
additional functions:
\begin{align*}
\hat{A}(x) &=\theta\big({-}x,-ax,-bx,-abx^3;q^2\big), \\
\check{A}(x) &=\theta\big({-}x,-ax,-bx,-abx^3/q^2;q^2\big).
\end{align*}

\begin{theorem}{\textup{\cite[Theorem~3.1]{Kim2026}}}\label{Main-THM-4}
Let $a,\!\:b,\!\:c,\!\:d,\!\:q\in\mathbb{C}\setminus\{0\}$ with $|q|<1$. Then
\begin{align}
 &\hat{A}(x_1)\check{A}(qx_1)\hat{A}(x_2)-\hat{A}(-x_1)\check{A}(-qx_1)\hat{A}(-x_2)
\nonumber\\
 &\quad-x_2\mathscr{S}(x_2)\big\{\hat{A}(x_1)\check{A}(qx_1)\check{A}(qx_2)
+\hat{A}(-x_1)\check{A}(-qx_1)\check{A}(-qx_2)\big\}\nonumber\\
 &\qquad=-\dfrac{2abd^3}{q}\theta\big({-}ac^2,-bc^2,-abc^2;q\big)
\theta\big(ad^2q,bd^2q,abd^2q;q^2\big).\label{Kim-iden-4}
\end{align}
\end{theorem}

The remainder of this paper is organized as follows. In Section~\ref{Sect:EP-1}, we
first establish an auxiliary identity and then give an elementary proof of Theorem
\ref{Main-THM-1}. Section \ref{Sect:EP-2} is devoted to the elementary proofs of
Theorems \ref{Main-THM-2}--\ref{Main-THM-4}.

\section{Elementary proof of Theorem~\ref{Main-THM-1}}\label{Sect:EP-1}

In this section, we provide an elementary $q$-series proof of Theorem~\ref{Main-THM-1}.
The proof consists of two main steps. First, using Warnaar's identity together with
elementary theta function transformations, we transform the left-hand side of
\eqref{Kim-iden-1} into a more tractable expression involving theta functions. We then
evaluate the resulting expression by means of Jacobi's triple product identity and
Weierstrass' fundamental theta identity. This ultimately yields \eqref{Kim-iden-1}.

\subsection{An auxiliary identity}\label{Subs:FS}
We first establish an auxiliary identity that will be used in the proof of
Theorem~\ref{Main-THM-1}.

\begin{lemma}\label{Fund-Lem-1}
Let
\begin{align}
H(X,Y)=\theta\big(qX;q^2\big)\theta\big({-}qY;q^2\big)+\theta\big({-}qX;q^2)
\theta\big(qY;q^2\big).\label{Def-H}
\end{align}
Then
\begin{align}
\mathscr{D}\big(x_1^2,x_2^2,x_3^2\big) &-\mathscr{S}\big(x_1^2,x_2^2,x_3^2\big)
\mathscr{D}\big(qx_1^2,qx_2^2,qx_3^2\big)\nonumber\\
 &\qquad\qquad=-\dfrac{c^2z^2}{q^2}\dfrac{\theta\big(a^2,b^2;q^2\big)}{\theta(a,b;q^2)}
C(r,z),\label{Fund-iden-1}
\end{align}
where $r=d/c$, $z=abc^2d^2$, and
\begin{align}
C(r,z) &=r^2\theta\big({-}z/r^2;q^2\big)H\big(zr^2,z\big)+\theta\big({-}zr^2;q^2\big)
H\big(z,z/r^2\big)\nonumber\\
 &\quad+r\!\:\theta\big({-}z;q^2\big)H\big(z/r^2,zr^2\big).\label{Def-C}
\end{align}
\end{lemma}

\begin{proof}
For any $x\neq0$, define
\begin{align}
D(x) &=A(x)-A(-x),\label{Def-D-x}\\
M(x) &=A(x)+\mathscr{S}(x)A(-qx),\label{Def-M-x}\\
N(x) &=\mathscr{S}(x)A(-qx)-A(-x).\label{Def-N-x}
\end{align}
We first derive explicit expressions for these functions. Replacing $q$ by $q^2$ and
setting $c=x$ in \eqref{War-gene}, we find that
\begin{align}
D(x)=2x(-q^2;q^2)_\infty^2\!\:\theta\big({-}a,-b,-abx^2;q^2\big).\label{D-x-exp}
\end{align}
We shall also use the elementary transformation
\begin{align}
\theta\big(q^2x;q^2\big)=-x^{-1}\theta\big(x;q^2\big),\label{Theta-prop}
\end{align}
which follows immediately from the definition of the theta function. Replacing $x$ by
$qx$ in \eqref{D-x-exp}, multiplying both sides by $\mathscr{S}(x)$ and then applying
\eqref{Theta-prop}, we deduce that
\begin{align*}
\mathscr{S}(x)D(qx) &=2qx(-q^2;q^2)_\infty^2\!\:\mathscr{S}(x)\!\:\theta
\big({-}a,-b,-abq^2x^2;q^2\big)
\nonumber\\
 &=2x(-q^2;q^2)_\infty^2\!\:\theta\big({-}a,-b,-abx^2;q^2\big).
\end{align*}
Consequently,
\begin{align}
\mathscr{S}(x)\big(A(qx)-A(-qx)\big)=D(x).\label{Aux-iden-1}
\end{align}

Next, applying \eqref{War-gene} with bases $q$ and $-q$, respectively, yields
\begin{align}
A(x)A(qx)-A(-x)A(-qx) &=2x(-q;q)_\infty^2\!\:\theta\big({-}a,-b,-abx^2;q\big),
\label{Sub-iden-1}\\
A(x)A(-qx)-A(-x)A(qx) &=2x(q;-q)_\infty^2\!\:\theta\big({-}a,-b,-abx^2;-q\big).
\label{Sub-iden-2}
\end{align}
We now rewrite the left-hand sides of \eqref{Sub-iden-1} and \eqref{Sub-iden-2} in
terms of $D$, $M$ and $N$. Indeed,
\begin{align}
 &\mathscr{S}(x)\big(A(x)A(qx)-A(-x)A(-qx)\big)\nonumber\\
 &\qquad\qquad\quad=\mathscr{S}(x)\big(A(qx)-A(-qx)\big)A(x)\nonumber\\
 &\qquad\qquad\qquad\qquad\qquad+\mathscr{S}(x)\big(A(x)-A(-x)\big)A(-qx).
\label{Der-step-1}
\end{align}
It follows from \eqref{Def-D-x}, \eqref{Def-M-x} and \eqref{Aux-iden-1} that
\begin{align}
\mathscr{S}(x)\big(A(x)A(qx)-A(-x)A(-qx)\big)=D(x)M(x).\label{Rel-iden-1}
\end{align}
Similarly,
\begin{align}
\mathscr{S}(x)\big(A(x)A(-qx)-A(-x)A(qx)\big)=D(x)N(x).\label{Rel-iden-2}
\end{align}
Combining \eqref{Sub-iden-1}, \eqref{Sub-iden-2}, \eqref{Rel-iden-1} and
\eqref{Rel-iden-2}, and simplifying the resulting theta products, we obtain that
\begin{align}
M(x) &=(-q;q^2)_\infty^2\!\:\mathscr{S}(x)\!\:\theta\big({-}aq,-bq,-abqx^2;q^2\big),
\label{Aux-iden-2}\\
N(x) &=(q;q^2)_\infty^2\!\:\mathscr{S}(x)\!\:\theta\big(aq,bq,abqx^2;q^2\big).
\label{Aux-iden-3}
\end{align}

The definitions \eqref{Def-D-x}--\eqref{Def-N-x}, together with \eqref{Aux-iden-1},
give
\begin{align}
A(x) &=\dfrac{M(x)+D(x)-N(x)}{2},\label{Rel-iden-3}\\
A(-x) &=\dfrac{M(x)-D(x)-N(x)}{2},\label{Rel-iden-4}\\
\mathscr{S}(x)A(qx) &=\dfrac{M(x)+D(x)+N(x)}{2},\label{Rel-iden-5}\\
\mathscr{S}(x)A(-qx) &=\dfrac{M(x)-D(x)+N(x)}{2}.\label{Rel-iden-6}
\end{align}
Substituting \eqref{Rel-iden-3}--\eqref{Rel-iden-6} into the left-hand side of
\eqref{Fund-iden-1}, we expand the products and collect like terms. Except for the
following six terms, all other terms cancel pairwise:
\begin{align}
 &\mathscr{D}\big(x_1^2,x_2^2,x_3^2\big)-\mathscr{S}\big(x_1^2,x_2^2,x_3^2\big)
\mathscr{D}\big(qx_1^2,qx_2^2,qx_3^2\big)\nonumber\\
 &\qquad\qquad=-\dfrac{1}{2}\sum_{(i,j,k)\in S}D\big(x_i^2\big)\big\{N\big(x_j^2\big)
M\big(x_k^2\big)+M\big(x_j^2\big)N\big(x_k^2\big)\big\},\label{LHS-sum-exp}
\end{align}
where
\begin{align*}
S=\{(1,2,3),\,(2,3,1),\,(3,1,2)\}.
\end{align*}
Finally, substituting \eqref{D-x-exp}, \eqref{Aux-iden-2}, and \eqref{Aux-iden-3} into
\eqref{LHS-sum-exp}, and collecting the common factors, we obtain a linear combination
of theta products. Using the elementary transformations of theta functions, this
expression simplifies to the right-hand side of \eqref{Fund-iden-1}. Hence,
\eqref{Fund-iden-1} follows.

This completes the proof of Lemma~\ref{Fund-Lem-1}.
\end{proof}

\subsection{Elementary proof of \texorpdfstring{\eqref{Kim-iden-1}}{}}
We now complete the elementary proof of \eqref{Kim-iden-1} by establishing the following
auxiliary identity.

\begin{lemma}
We have
\begin{align}
C(r,z)=2\dfrac{\theta\big(r^3,z^2;q^2\big)}{\theta\big(r,z;q^2\big)}.\label{Fund-iden-2}
\end{align}
\end{lemma}

\begin{proof}
We first derive a theta function representation for $H(X,Y)$. Namely, we claim that
\begin{align}
H(X,Y)=2(-q^2;q^2)_\infty^2\!\:\theta\big(q^2XY,q^2X/Y;q^4\big).\label{H-theta-exp}
\end{align}
We use Jacobi's celebrated triple product identity \cite[p.~15, Eq.~(1.6.1)]{GR2004}
\begin{align}
(x;q)_\infty(q/x;q)_\infty(q;q)_\infty=\sum_{n=-\infty}^\infty(-1)^nx^nq^{n(n-1)/2}.
\label{JTPI}
\end{align}
Upon taking $x=qX$ and $x=-qX$, respectively, in \eqref{JTPI}, we find that
\begin{align}
(q^2;q^2)_\infty\!\:\theta(qX;q^2) &=\sum_{n=-\infty}^\infty(-1)^nX^nq^{n^2},
\label{Prod-sum-1}\\
(q^2;q^2)_\infty\!\:\theta(-qX;q^2) &=\sum_{n=-\infty}^\infty X^nq^{n^2}.
\label{Prod-sum-2}
\end{align}
Substituting \eqref{Prod-sum-1} and \eqref{Prod-sum-2} into the definition
\eqref{Def-H}, we deduce that
\begin{align}
(q^2;q^2)_\infty^2\!\:H(X,Y)=\sum_{m=-\infty}^\infty\sum_{n=-\infty}^\infty
\big((-1)^{m}+(-1)^n\big)X^{m}Y^{n}q^{m^2+n^2}.\label{H-exp-1}
\end{align}
The factor $(-1)^m+(-1)^n$ vanishes whenever $m$ and $n$ have opposite parity. Thus,
only the terms with $m\equiv n\pmod 2$ contribute. For such pairs $(m,n)$, write
$m=s+t$, $n=s-t$, where $s,\!\:t\in\mathbb{Z}$. This change of variables is a bijection
between the pairs $(m,n)$ with $m\equiv n\pmod{2}$ and the
pairs $(s,t)\in\mathbb{Z}^2$. Consequently, \eqref{H-exp-1} becomes
\begin{align}
(q^2;q^2)_\infty^2 &\!\:H(X,Y)\nonumber\\
 &=2\sum_{s=-\infty}^\infty\big({-}q^2XY\big)^{s}q^{2s(s-1)}\sum_{t=-\infty}^\infty
\big({-}q^2X/Y\big)^tq^{2t(t-1)}.\label{H-exp-2}
\end{align}
Applying Jacobi's triple product identity \eqref{JTPI} to each of the two single sums
in \eqref{H-exp-2}, we obtain that
\begin{align*}
H(X,Y) &=2\dfrac{(q^4;q^4)_\infty^2}{(q^2;q^2)_\infty^2}\!\:\theta\big(q^2XY,q^2X/Y;
q^4\big)\nonumber\\
 &=2(-q^2;q^2)_\infty^2\!\:\theta\big(q^2XY,q^2X/Y;q^4\big).
\end{align*}
This proves \eqref{H-theta-exp}.

We next establish the identity
\begin{align}
(-q^2;q^2)_\infty^2\Bigg\{\theta\big(r^2q^2;q^4\big)\dfrac{\theta\big(r^4;q^4\big)}
{\theta\big(r^2;q^4\big)}+r\!\:\theta\big(r^4q^2;q^4\big)\Bigg\}
=\dfrac{\theta\big(r^3;q^2\big)}{\theta\big(r;q^2\big)}.\label{Key-iden-1}
\end{align}
Since $\theta\big(r;q^2\big)\theta\big({-}r;q^2\big)=\theta\big(r^2;q^4\big)$,
\eqref{Key-iden-1} is equivalent to
\begin{align}
 &\theta\big({-}r;q^2\big)\theta\big(r^3;q^2\big)\nonumber\\
 &\qquad=(-q^2;q^2)_\infty^2\big\{\theta\big(r^2q^2;q^4\big)\theta\big(r^4;q^4\big)
+r\!\:\theta\big(r^2;q^4\big)\theta\big(r^4q^2;q^4\big)\big\}.\label{Key-iden-2}
\end{align}
By \eqref{JTPI}, we have
\begin{align*}
(q^2;q^2)_\infty^2\!\:\theta\big({-}r;q^2\big)\theta\big(r^3;q^2\big)
=\sum_{m=-\infty}^\infty\sum_{n=-\infty}^\infty(-1)^nr^{m+3n}q^{m(m-1)+n(n-1)}.
\end{align*}
We now split the double sum according to the parity of $m+n$. When $m+n$ is even, we
use the change of variables $(m,n)=(t-s,s+t)$; when $m+n$ is odd, we use
$(m,n)=(s-t+1,s+t)$. Both changes of variables are bijections between the corresponding
pairs of integers. After a straightforward simplification, the resulting double sum can
be written as
\begin{align}
(q^2;q^2)_\infty^2\!\: &\theta\big({-}r;q^2\big)\theta\big(r^3;q^2\big)\nonumber\\
 &=\sum_{s=-\infty}^\infty(-q^2r^2)^{s}q^{2s(s-1)}\sum_{t=-\infty}^\infty(-r^4)^{t}
q^{2t(t-1)}\nonumber\\
 &\quad+r\sum_{s=-\infty}^\infty(-q^2r^4)^{s}q^{2s(s-1)}\sum_{t=-\infty}^\infty
(-r^2)^{t}q^{2t(t-1)}.\label{Key-iden-3}
\end{align}
Applying \eqref{JTPI} separately to the four single sums on the right-hand side of
\eqref{Key-iden-3}, we deduce that
\begin{align}
 &(q^2;q^2)_\infty^2\!\:\theta\big({-}r;q^2\big)\theta\big(r^3;q^2\big)\nonumber\\
 &\qquad=(q^4;q^4)_\infty^2\big\{\theta\big(q^2r^2;q^4\big)\theta\big(r^4;q^4\big)
+r\!\:\theta\big(r^2;q^4\big)\theta\big(q^2r^4;q^4\big)\big\}.\label{Key-iden-4}
\end{align}
Dividing both sides of \eqref{Key-iden-4} by $(q^2;q^2)_\infty^2$, we recover
\eqref{Key-iden-2}, and hence \eqref{Key-iden-1}.

We now turn to \eqref{Def-C}. By \eqref{H-theta-exp}, we have
\begin{align*}
C(r,z)=2(-q^2;q^2)_\infty^2\!\:S(r,z),
\end{align*}
where
\begin{align}
S(r,z) &=r^2\theta\big({-}z/r^2;q^2\big)\theta\big(z^2r^2q^2,r^2q^2;q^4\big)
\nonumber\\
 &\qquad\qquad\quad+\theta\big({-}zr^2;q^2\big)\theta\big(z^2q^2/r^2,r^2q^2;q^4\big)
\nonumber\\
 &\qquad\qquad\qquad\qquad\quad+r\!\:\theta\big({-}z;q^2\big)
\theta\big(z^2q^2,r^4q^2;q^4\big).\label{Def-S}
\end{align}
We simplify $S(r,z)$ by applying Weierstrass' fundamental theta identity
\cite{Wei1882}; see also \cite[Eq.~(11.4.3)]{GR2004}:
\begin{align}
\theta(xy,x/y,uv,u/v;q) &-\theta(xv,x/v,uy,u/y;q)\nonumber\\
 &\qquad\qquad=\dfrac{u}{y}\!\;\theta(xu,x/u,yv,y/v;q).\label{Wei-iden}
\end{align}
Replacing $q$ by $q^4$, and then making the substitution
\begin{align*}
(x,y,u,v)\mapsto(zr^2q,zq/r^2,zq,-1/q),
\end{align*}
we find that
\begin{align}
\theta\big(z^2q^2,r^4,-z,-zq^2;q^4\big) &-\theta\big({-}zr^2,-zr^2q^2,z^2q^2/r^2,r^2;
q^4\big)\nonumber\\
 &\quad=r^2\!\:\theta\big(z^2r^2q^2,r^2,-z/r^2,-zq^2/r^2;q^4\big).\label{Der-iden-1}
\end{align}
Dividing both sides of \eqref{Der-iden-1} by $\theta\big(r^2;q^4\big)$, we deduce that
\begin{align}
r^2\!\:\theta\big({-}z/r^2;q^2\big)\theta\big(z^2r^2q^2;q^4\big)
 &+\theta\big({-}zr^2;q^2\big)\theta\big(z^2q^2/r^2;q^4\big)\nonumber\\
 &\quad=\dfrac{\theta\big(r^4;q^4\big)}{\theta\big(r^2;q^4\big)}\theta\big({-}z;q^2\big)
\theta\big(z^2q^2;q^4\big).\label{Der-iden-2}
\end{align}
Substituting \eqref{Der-iden-2} into \eqref{Def-S} and then applying
\eqref{Key-iden-1}, we conclude that
\begin{align*}
C(r,z)=2\dfrac{\theta\big(r^3,z^2;q^2\big)}{\theta\big(r,z;q^2\big)}.
\end{align*}
Thus, \eqref{Fund-iden-2} follows.
\end{proof}

The general theta function identity \eqref{Kim-iden-1} now follows immediately by
combining \eqref{Fund-iden-1} and \eqref{Fund-iden-2}, together with \eqref{Theta-prop}.

\section{Elementary proofs of Theorems~\ref{Main-THM-2}--\ref{Main-THM-4}}
\label{Sect:EP-2}

This section is devoted to the elementary proofs of
Theorems~\ref{Main-THM-2}--\ref{Main-THM-4}. The main ingredients used in these proofs
were developed in subsection~\ref{Subs:FS} during the proof of \eqref{Kim-iden-1}. In
particular, the auxiliary identities established there also serve as the principal
tools in the proofs of Theorems~\ref{Main-THM-2}--\ref{Main-THM-4}. Thus, the four
identities \eqref{Kim-iden-1}--\eqref{Kim-iden-4} admit elementary $q$-series proofs
within a common framework.

\subsection{Elementary proof of Theorem~\ref{Main-THM-2}}
By \eqref{Rel-iden-3}--\eqref{Rel-iden-6}, we obtain that
\begin{align}
\mathscr{D}(x_1,x_2)+\mathscr{S}(x_1,x_2)\mathscr{D}(qx_1,qx_2)=D(x_1)M(x_2)
+M(x_1)D(x_2).\label{PF-step-2-1}
\end{align}
Substituting \eqref{Def-D-x} and \eqref{Def-M-x} into the right-hand side of
\eqref{PF-step-2-1}, we deduce that
\begin{align}
D(x_1)M(x_2) &+M(x_1)D(x_2)\nonumber\\
 &=\dfrac{2abc^2d}{q}(-q;q)_\infty^2\!\:\theta(-a,-b;q)\nonumber\\
 &\quad\times\big\{\theta\big({-}qy/r,-yr;q^2\big)
+r\!\:\theta\big({-}y/r,-qyr;q^2\big)\big\},\label{PF-step-2-2}
\end{align}
where $r=d/c$ and $y=abcd$.

It remains to evaluate the theta function combination in braces. For this purpose,
applying \eqref{JTPI}, we have
\begin{align*}
(q;q)_\infty^2\!\:\theta(-y,-r;q)=\sum_{m=-\infty}^\infty\sum_{n=-\infty}^\infty y^{m}
r^{n}q^{m(m-1)/2+n(n-1)/2}.
\end{align*}
We split this double sum according to the parity of $m+n$. When $m+n$ is even, we make
the change of variables $(m,n)=(s+t,s-t)$, which gives a bijection from $\mathbb{Z}^2$ onto the even sublattice $\{(m,n)\in\mathbb{Z}^2\colon m+n\equiv0\pmod{2}\}$. When
$m+n$ is odd, we instead use $(m,n)=(s+t,s-t+1)$, which is a bijection from $\mathbb{Z}^2$ onto the odd sublattice
$\{(m,n)\in\mathbb{Z}^2\colon m+n\equiv1\pmod{2}\}$. Under these two changes of
variables, the double sum becomes
\begin{align}
(q;q)_\infty^2\!\:\theta(-r,-y;q) &=\sum_{s=-\infty}^\infty(yr)^sq^{s(s-1)}
\sum_{t=-\infty}^\infty(qy/r)^tq^{t(t-1)}\nonumber\\
 &\quad+r\sum_{s=-\infty}^\infty(qyr)^sq^{s(s-1)}\sum_{t=-\infty}^\infty(y/r)^t
q^{t(t-1)}.\label{PF-step-2-3}
\end{align}
Applying separately to the four single sums in \eqref{PF-step-2-3}, we obtain that
\begin{align}
\theta(-r,-y;q)=(-q;q)_\infty^2\big\{\theta\big({-}qy/r,-yr;q^2\big)
+r\!\:\theta\big({-}y/r,-qyr;q^2\big)\big\}.\label{PF-step-2-4}
\end{align}
Finally, substituting \eqref{PF-step-2-4} into \eqref{PF-step-2-2}, we conclude that
\begin{align*}
\mathscr{D}(x_1,x_2) &+\mathscr{S}(x_1,x_2)\mathscr{D}(qx_1,qx_2)\\
 &\quad=\dfrac{2abc^2d}{q}\theta(-a,-b,-d/c,-abcd;q).
\end{align*}

\subsection{Elementary proof of Theorem~\ref{Main-THM-3}}
By \eqref{Def-D-x}, \eqref{Def-N-x} and \eqref{Aux-iden-1}, we have
\begin{align}
A(x)-\mathscr{S}(x)A(qx)=A(-x)-\mathscr{S}(x)A(-qx)=-N(x).\label{Comb-iden}
\end{align}
It follows from \eqref{Comb-iden} that
\begin{align}
\mathscr{D}(x_1,qx_1,x_2)-\mathscr{S}(x_2)\mathscr{D}(x_1,qx_1,qx_2)
=-\mathscr{D}(x_1,qx_1)N(x_2).\label{PF-step-3-1}
\end{align}

On the other hand, applying Warnaar's identity \eqref{War-gene} with $c=x$, we find that
\begin{align}
\mathscr{D}(x,qx)=2x(-q;q)_\infty^2\!\:\theta\big({-}a,-b,-abx^2;q\big).
\label{PF-step-3-2}
\end{align}
Substituting \eqref{Aux-iden-3} and \eqref{PF-step-3-2} into \eqref{PF-step-3-1}, with
$x$ replaced by $x_1$ only in \eqref{PF-step-3-2}, and simplifying, we conclude that
\begin{align*}
\mathscr{D}(x_1,qx_1,x_2) &-\mathscr{S}(x_2)\mathscr{D}(x_1,qx_1,qx_2)\nonumber\\
 &\quad=-\dfrac{2abcd^2}{q}\dfrac{\theta\big(a^2,b^2,a^2b^2c^4,abd^2q;q^2\big)}
{\theta\big(a,b,abc^2,abc^2q;q^2\big)}.
\end{align*}

\subsection{Elementary proof of Theorem~\ref{Main-THM-4}}
First, making the parameter substitution $(a,b,x)\mapsto(ac^2,bc^2,1/c)$ in
\eqref{PF-step-3-2} and simplifying, we find that
\begin{align}
\hat{A}(x_1)\check{A}(qx_1)+\hat{A}(-x_1)\check{A}(-qx_1)=2(-q;q)_\infty^2\!\:\theta
\big({-}ac^2,-bc^2,-abc^2,q\big).\label{PF-step-4-1}
\end{align}

Next, substituting \eqref{Aux-iden-3} into \eqref{Comb-iden}, we deduce that
\begin{align}
A(x)-\mathscr{S}(x)A(qx) &=-(q;q^2)_\infty^2\!\:\mathscr{S}(x)\!\:\theta\big(aq,
bq,abqx^2;q^2\big),\label{PF-step-4-2}\\
A(-x)-\mathscr{S}(x)A(-qx) &=-(q;q^2)_\infty^2\!\:\mathscr{S}(x)\!\:\theta\big(aq,
bq,abqx^2;q^2\big).\label{PF-step-4-3}
\end{align}
Making the parameter substitution $(a,b,x)\mapsto(ad^2,bd^2,1/d)$ in
\eqref{PF-step-4-2} and \eqref{PF-step-4-3}, respectively, and simplifying, we obtain
that
\begin{align}
\hat{A}(x_2) &-x_2\mathscr{S}(x_2)\check{A}(qx_2)\nonumber\\
 &\quad=-x_2\mathscr{S}(x_2)(q;q^2)_\infty^2\!\:\theta\big(ad^2q,bd^2q,abd^2q;q^2\big),
\label{PF-step-4-4}\\
\hat{A}(-x_2) &+x_2\mathscr{S}(x_2)\check{A}(-qx_2)\nonumber\\
 &\quad=x_2\mathscr{S}(x_2)\!\:(q;q^2)_\infty^2\theta\big(ad^2q,bd^2q,abd^2q;q^2\big).
\label{PF-step-4-5}
\end{align}
Combining \eqref{PF-step-4-1}, \eqref{PF-step-4-4} and \eqref{PF-step-4-5}, we conclude
that
\begin{align*}
 &\hat{A}(x_1)\check{A}(qx_1)\hat{A}(x_2)-\hat{A}(-x_1)\check{A}(-qx_1)\hat{A}(-x_2)\\
 &\quad-x_2\mathscr{S}(x_2)\big\{\hat{A}(x_1)\check{A}(qx_1)\check{A}(qx_2)
+\hat{A}(-x_1)\check{A}(-qx_1)\check{A}(-qx_2)\big\}\nonumber\\
 &\qquad=\hat{A}(x_1)\check{A}(qx_1)\big\{\hat{A}(x_2)-x_2\mathscr{S}(x_2)
\check{A}(x_2q)\big\}\\
 &\qquad\qquad-\hat{A}(-x_1)\check{A}(-qx_1)\big\{\hat{A}(-x_2)+x_2\mathscr{S}(x_2)
\check{A}(-x_2q)\big\}\\
 &\qquad=-\dfrac{2abd^3}{q}\theta\big({-}ac^2,-bc^2,-abc^2;q\big)
\theta\big(ad^2q,bd^2q,abd^2q;q^2\big).
\end{align*}

\section{Closing remarks}

Kim \cite[Theorem~6.1]{Kim2026} also proved the following general theta function
identity:
\begin{align}
\theta\big({-}acq,-acq,-a^2c,-a^2c;q^2\big) &-a\!\:\theta\big({-}ac,-ac,-a^2cq,-a^2cq;
q^2\big)\nonumber\\
 &\quad=\theta\big(a,a^3c^2,q,q;q^2\big).\label{Kim-iden-5}
\end{align}
In fact, \eqref{Kim-iden-5} follows directly from Weierstrass' fundamental theta
identity \eqref{Wei-iden}. More precisely, replacing $q$ by $q^2$ in \eqref{Wei-iden}
and making the parameter substitution
\begin{align*}
(x,y,u,v)\mapsto(a^{3/2}c\!\:q^{1/2},-a^{-1/2}q^{1/2},-a^{1/2}q^{1/2},a^{-3/2}c^{-1}
q^{1/2}),
\end{align*}
yields \eqref{Kim-iden-5} after simplification.

\section*{Acknowledgements}

The author is sincerely grateful to his colleague, Dr.~Lichun Liang, for his generous
assistance throughout this research, and Professor Sun Kim for her valuable comments
and suggestions on an earlier version of this manuscript.
This work was partially supported by the National Natural Science Foundation of China
(No.~12201093) and the Science and Technology Research Program of Chongqing Municipal
Education Commission (No.~KJQN202500501).

\section*{Declaration of AI usage}

During the development of this work, the author used ChatGPT (OpenAI, GPT-5.6 Pro)
as an AI-assisted research tool to explore possible proof strategies for
Theorem~\ref{Main-THM-1}. All suggestions generated by the tool were critically
evaluated and independently verified by the author. The resulting arguments were substantially revised, simplified and reformulated by the author, who takes full
responsibility for the accuracy, originality and integrity of the mathematical content
of this work.

\bibliographystyle{amsplain}

\end{document}